\documentclass[11pt]{article}
\usepackage[T1]{fontenc}
\usepackage[utf8]{inputenc}
\usepackage{lmodern}
\usepackage{amsmath,amssymb,amsthm,mathtools}
\usepackage[margin=1in]{geometry}
\usepackage{microtype}
\usepackage{tikz}
\usepackage{float}
\usepackage{enumerate}
\usepackage{indentfirst}
\usepackage[colorlinks=true,linkcolor=blue,citecolor=blue,urlcolor=blue]{hyperref}

\theoremstyle{plain}
\newtheorem{theorem}{Theorem}[section]
\newtheorem{lemma}[theorem]{Lemma}
\newtheorem{question}[theorem]{Question}
\theoremstyle{remark}

\DeclareMathOperator{\dist}{dist}

\tikzset{
	vertex/.style={circle,fill=black,inner sep=1.7pt},
	cycleedge/.style={line width=1.25pt},
	outeredge/.style={bend right=32}
}
\newcommand{\gadgetcoordinates}{%
	\coordinate (a) at (0,2.05);
	\coordinate (b) at (-1.75,-1.05);
	\coordinate (c) at (1.75,-1.05);
	\coordinate (x) at (-0.72,0.73);
	\coordinate (y) at (0,-0.63);
	\coordinate (z) at (0.72,0.73);
	\coordinate (w) at (0,0.27);
}
\newcommand{\outertriangle}[1]{%
	\draw[#1] (a) to[outeredge] (b) to[outeredge] (c) to[outeredge] (a);
}
\newcommand{\gadgetedges}[1]{%
	\outertriangle{#1}
	\draw[#1] (a)--(x) (b)--(x) (b)--(y) (c)--(y) (c)--(z) (a)--(z);
	\draw[#1] (x)--(y)--(z)--cycle;
	\draw[#1] (w)--(x) (w)--(y) (w)--(z);
}
\newcommand{\gadgetlabels}{%
	\node[vertex,label=above:$a$] at (a) {};
	\node[vertex,label=below left:$b$] at (b) {};
	\node[vertex,label=below right:$c$] at (c) {};
	\node[vertex,label=left:$x$] at (x) {};
	\node[vertex,label=below:$y$] at (y) {};
	\node[vertex,label=right:$z$] at (z) {};
	\node[vertex,label=above:$w$] at (w) {};
}

\hypersetup{
	pdftitle={Non-Hamiltonian 3/2-tough plane triangulations},
	pdfauthor={Songling Shan},
	pdfkeywords={Hamiltonian cycle, plane triangulation, toughness, 2-factor}
}

\begin{document}

\title{Non-Hamiltonian $\frac{3}{2}$-Tough Plane Triangulations}
\author{
	Songling Shan\thanks{Department of Mathematics and Statistics,
		Auburn University, Auburn, AL 36849.
		Email: \texttt{szs0398@auburn.edu}.
		Partially supported by NSF grant DMS-2451895.}
}
\date{\today}
\maketitle

\begin{abstract}
	By Tutte's classic theorem of 1956 that every 4-connected planar graph is
	Hamiltonian, every planar graph of order at least three with toughness
	greater than $\frac{3}{2}$ is Hamiltonian. In 1999, Owens constructed a
	sequence of maximal planar graphs whose toughness approaches $\frac{3}{2}$
	from below and which do not contain even a 2-factor, and he asked whether
	there exists a maximal planar graph with toughness exactly $\frac{3}{2}$
	and with no 2-factor. In 2025, Shan constructed a $\frac{3}{2}$-tough plane
	triangulation with no 2-factor. In that construction, there are many pairs
	of vertices of degree $3$ that have a common neighbor. By imposing a
	distance condition on the vertices of degree $3$, Hao, Ma, Shan, and Yang
	recently proved that every $\frac{3}{2}$-tough plane triangulation of order
	at least three whose vertices of degree $3$ are pairwise at distance at
	least $3$ has a 2-factor, and they asked whether every such graph is
	Hamiltonian. We answer this question in the negative, and in fact prove the
	following stronger statement: for every positive integer $\ell$, there
	exists a $\frac{3}{2}$-tough non-Hamiltonian plane triangulation whose
	vertices of degree $3$ are pairwise at distance at least $\ell$. Thus,
	although the distance condition guarantees the existence of a 2-factor, it
	does not guarantee that the graph is Hamiltonian: the essential obstruction
	to a Hamiltonian cycle is a certain local configuration involving a vertex
	of degree $3$, rather than the proximity of such configurations in the
	graph.
\end{abstract}


\section{Introduction}

We consider only finite simple graphs. Let $G$ be a graph. Denote by $V(G)$
and $E(G)$ the vertex set and the edge set of $G$, respectively, and by
$c(G)$ the number of components of $G$. For a vertex $v$ of $G$, let
$N_G(v)$ be the set of neighbors of $v$ and $d_G(v)=|N_G(v)|$ the degree of
$v$. For $X\subseteq V(G)$, let $G[X]$ be the subgraph of $G$ induced by
$X$, and let $G-X=G[V(G)\setminus X]$; we write $G-v$ for $G-\{v\}$. For
two vertices $u,v\in V(G)$, let $\dist_G(u,v)$ be the length of a shortest
path connecting $u$ and $v$ in $G$. A \emph{plane triangulation} is a plane
graph of order at least three in which every face is bounded by a triangle;
plane triangulations are precisely the plane embeddings of maximal planar
graphs of order at least three.

Let $t\ge 0$ be a real number. A graph $G$ is \emph{$t$-tough} if
$|S|\ge t\,c(G-S)$ for every $S\subseteq V(G)$ with $c(G-S)\ge 2$. The
\emph{toughness} $\tau(G)$ of a noncomplete graph $G$ is the largest real
number $t$ for which $G$ is $t$-tough; as usual, $\tau(K_n)=\infty$. The
concept of toughness was introduced by Chv\'atal~\cite{Chvatal1973}.

A \emph{2-factor} of a graph is a spanning 2-regular subgraph. Thus, a
Hamiltonian cycle is a 2-factor with exactly one component. A graph with
toughness greater than $\frac{3}{2}$ has no vertex cut of size at most
three, so by Tutte's theorem that every 4-connected planar graph is
Hamiltonian~\cite{Tutte1956}, every planar graph of order at least three
with toughness greater than $\frac{3}{2}$ is Hamiltonian.

In 1999, Owens~\cite{Owens1999} constructed a sequence of maximal planar
graphs whose toughness approaches $\frac{3}{2}$ from below and which do not
contain even a 2-factor. He then asked whether there exists a maximal planar
graph with toughness exactly $\frac{3}{2}$ and with no 2-factor. Bauer,
Broersma, and Schmeichel~\cite{Bauer2006} later wrote: ``One of the
challenging open problems in this area is to determine whether every
$\frac{3}{2}$-tough maximal planar graph has a 2-factor. If so, are they
all Hamiltonian? We also do not know if a $\frac{3}{2}$-tough planar graph
has a 2-factor.''

Answering these questions, Shan~\cite{Shan2025} constructed a
$\frac{3}{2}$-tough plane triangulation with no 2-factor. In that
construction, there are many pairs of vertices of degree $3$ that have a
common neighbor. By imposing a distance condition on vertices of degree
$3$, Hao, Ma, Shan, and Yang~\cite{HaoMaShanYang} recently proved the
following result.

\begin{theorem}[Hao, Ma, Shan, and Yang~\cite{HaoMaShanYang}]
	\label{thm:twofactor}
	Let $G$ be a $\frac{3}{2}$-tough plane triangulation of order at least
	three. If $\dist_G(u,v)\ge 3$ for every two distinct vertices $u$ and
	$v$ of degree $3$, then $G$ has a 2-factor.
\end{theorem}

They then posed the corresponding question for Hamiltonian cycles.

\begin{question}[Hao, Ma, Shan, and Yang~\cite{HaoMaShanYang}]
	\label{ques:hcyc}
	Is every $\frac{3}{2}$-tough plane triangulation of order at least three
	Hamiltonian if the distance between every two vertices of degree $3$ is
	at least $3$?
\end{question}

We answer this question in the negative. In fact, the lower bound on the
distance can be made arbitrarily large.

\begin{theorem}\label{thm:main}
	For every positive integer $\ell$, there exists a plane triangulation
	$G_\ell$ with $\tau(G_\ell)=\frac{3}{2}$ such that the following
	statements hold:
	\begin{enumerate}[(i)]
		\item
		$\dist_{G_\ell}(u,v)\ge \ell$ for every two distinct vertices $u$
		and $v$ of degree $3$;
		\item
		$G_\ell$ has a 2-factor; and
		\item
		$G_\ell$ is not Hamiltonian.
	\end{enumerate}
\end{theorem}

Thus, although the distance condition guarantees the existence of a
2-factor, it does not guarantee that the graph is Hamiltonian. The
essential obstruction to a Hamiltonian cycle is a certain local
configuration involving a vertex of degree $3$, rather than the proximity
of such configurations in the graph.

To prove Theorem~\ref{thm:main}, we introduce an operation that we call
\emph{tetrahedral replacement}. Let $G$ be a plane triangulation, and let
$v\in V(G)$ be a vertex of degree $3$ with $N_G(v)=\{a,b,c\}$. The three
faces of $G$ incident with $v$ are the triangles $vabv$, $vbcv$, and $vcav$;
in particular, $a,b,c$ induce a triangle in $G$. Delete $v$, add four new
vertices $x,y,z,w$, and add edges so that $\{x,y,z,w\}$ induces a copy of
$K_4$ and
\[
N(x)\cap\{a,b,c\}=\{a,b\},\qquad
N(y)\cap\{a,b,c\}=\{b,c\},\qquad
N(z)\cap\{a,b,c\}=\{c,a\}.
\]
No other edges incident with $x,y,z,w$ are added. Placing $x,y,z$ inside the
triangle $abc$ and $w$ inside the triangle $xyz$, as in
Figure~\ref{fig:replacement}, we obtain a plane graph, which we denote by
$R_v(G)$. We call this operation the \emph{tetrahedral replacement} at $v$,
and we call $w$ the \emph{successor} of $v$. Note that
$R_v(G)-\{x,y,z,w\}=G-v$.

\begin{figure}[ht]
	\centering
	\begin{minipage}[c]{0.38\textwidth}
		\centering
		\begin{tikzpicture}[scale=1.02]
			\gadgetcoordinates
			\coordinate (v) at (0,0.05);
			\outertriangle{}
			\draw (v)--(a) (v)--(b) (v)--(c);
			\node[vertex,label=above:$a$] at (a) {};
			\node[vertex,label=below left:$b$] at (b) {};
			\node[vertex,label=below right:$c$] at (c) {};
			\node[vertex,label=right:$v$] at (v) {};
		\end{tikzpicture}

		\small (a) The vertex $v$ in $G$.
	\end{minipage}
	\hfill
	\begin{minipage}[c]{0.55\textwidth}
		\centering
		\begin{tikzpicture}[scale=1.02]
			\gadgetcoordinates
			\gadgetedges{}
			\gadgetlabels
		\end{tikzpicture}

		\small (b) The vertices $x,y,z,w$ in $R_v(G)$.
	\end{minipage}
	\caption{The tetrahedral replacement at a vertex $v$ of degree $3$. All
		edges outside the triangle $abc$ remain unchanged.}
	\label{fig:replacement}
\end{figure}

The tetrahedral replacement has the following properties.

\begin{theorem}[Tetrahedral Replacement Theorem]
	\label{thm:tetrahedral-replacement}
	Let $G$ be a plane triangulation, let $v\in V(G)$ be a vertex of degree
	$3$, and let $G'=R_v(G)$. Let $w$ be the successor of $v$. Then the
	following statements hold:
	\begin{enumerate}[(i)]
		\item $G'$ is a plane triangulation;
		\item if $G$ is $\frac{3}{2}$-tough, then $G'$ is
		$\frac{3}{2}$-tough;
		\item $G'$ is Hamiltonian if and only if $G$ is Hamiltonian; and
		\item for all $p,q\in V(G)\setminus\{v\}$,
		\[
		\dist_{G'}(p,q)=\dist_{G}(p,q)
		\qquad\text{and}\qquad
		\dist_{G'}(w,q)=\dist_{G}(v,q)+1.
		\]
	\end{enumerate}
\end{theorem}

In Section~\ref{sec:counterexample} we deduce Theorem~\ref{thm:main} from
Theorem~\ref{thm:tetrahedral-replacement}, which is then proved in
Section~\ref{sec:replacement}. In Section~\ref{sec:twofactors} we explain
why the argument for Hamiltonian cycles in
Theorem~\ref{thm:tetrahedral-replacement}(iii) has no analogue for
2-factors.

\section{Proof of Theorem~\ref{thm:main}}\label{sec:counterexample}

For a graph $G$, let $D_3(G)$ denote the set of vertices of degree $3$ in
$G$. We first observe that in a plane triangulation other than $K_4$, the
vertices of degree $3$ are pairwise nonadjacent.

\begin{lemma}\label{lem:independent}
	Let $G$ be a plane triangulation with $G\ne K_4$. Then no two vertices of
	degree $3$ in $G$ are adjacent. Consequently, $\dist_G(u,v)\ge 2$ for
	any two distinct vertices $u,v\in D_3(G)$.
\end{lemma}

\begin{proof}
	Suppose to the contrary that $u,v\in D_3(G)$ are adjacent. Then
	$|V(G)|\ge 4$, so every edge of $G$ lies on exactly two faces. Let $a$
	and $b$ be the third vertices of the two faces incident with the edge
	$uv$. Then $a\ne b$, since otherwise $u$ would have only the two
	neighbors $v$ and $a$. As $d_G(u)=d_G(v)=3$, we have $N_G(u)=\{v,a,b\}$
	and $N_G(v)=\{u,a,b\}$. Since any two consecutive neighbors of a vertex
	in a plane triangulation form a face together with that vertex, the
	faces incident with $u$ or $v$ are exactly
	\[
	uvau,\quad uvbu,\quad uabu,\quad vabv.
	\]
	Each of these four triangles bounds a face of $G$, which is necessarily
	the side of the triangle not containing the fourth vertex of
	$\{u,v,a,b\}$. These four regions are precisely the four faces of the
	plane graph $G[\{u,v,a,b\}]\cong K_4$, and they cover every point of the
	plane not on this $K_4$. Hence $V(G)=\{u,v,a,b\}$ and $G=K_4$, a
	contradiction.
\end{proof}

Let $G\ne K_4$ be a plane triangulation, and let
$D_3(G)=\{v_1,\ldots,v_m\}$. By Lemma~\ref{lem:independent}, the vertices
$v_1,\ldots,v_m$ are pairwise nonadjacent. The tetrahedral replacement at
$v_i$ changes $G$ only inside the three faces incident with $v_i$, and it
does not change the degree or the neighborhood of any vertex outside
$N_G(v_i)\cup\{v_i\}$. Since $v_j\notin N_G(v_i)$ for $j\ne i$, the vertex
$v_j$ still has degree $3$ and the same three neighbors after the
replacement at $v_i$. Hence the tetrahedral replacements at
$v_1,\ldots,v_m$ can be performed one after another, and the resulting
graph does not depend on the order in which they are performed. We denote
this graph by $\mathcal{R}(G)$; if $m=0$, then $\mathcal{R}(G)=G$. For
$1\le i\le m$, let $w_i$ denote the successor of $v_i$, so that
$t_1,\ldots,w_m$ are vertices of $\mathcal{R}(G)$.

\begin{lemma}\label{lem:iterate}
	Let $G\ne K_4$ be a plane triangulation with $D_3(G)=\{v_1,\ldots,v_m\}$,
	and let $w_i$ be the successor of $v_i$ for $1\le i\le m$. Then the
	following statements hold:
	\begin{enumerate}[(a)]
		\item $\mathcal{R}(G)$ is a plane triangulation, which is
		$\frac{3}{2}$-tough if $G$ is $\frac{3}{2}$-tough, and which is
		Hamiltonian if and only if $G$ is Hamiltonian;
		\item $D_3(\mathcal{R}(G))=\{t_1,\ldots,w_m\}$; and
		\item $\dist_{\mathcal{R}(G)}(w_i,w_j)=\dist_G(v_i,v_j)+2$ for all
		distinct $i,j\in\{1,\ldots,m\}$.
	\end{enumerate}
\end{lemma}

\begin{proof}
	Statement (a) follows from $m$ applications of
	Theorem~\ref{thm:tetrahedral-replacement}(i)--(iii).

	For (b), consider the replacement at $v_i$, performed on the graph
	obtained from $G$ by the replacements at $v_1,\ldots,v_{i-1}$. By the
	discussion preceding the lemma, $v_i$ has the same three neighbors in
	this graph as in $G$, and their degrees are at least their degrees in
	$G$, which are at least $4$: every vertex of a plane triangulation of
	order at least four has degree at least $3$, and the neighbors of $v_i$
	do not have degree $3$ by Lemma~\ref{lem:independent}. The
	replacement at $v_i$ deletes $v_i$, creates one new vertex $w_i$ of
	degree $3$ and three new vertices of degree $5$, increases the degrees
	of the three neighbors of $v_i$ by one, and leaves all other degrees
	unchanged. Hence, after all $m$ replacements, the vertices of degree $3$
	are exactly $t_1,\ldots,w_m$.

	For (c), fix distinct $i$ and $j$, and perform the replacements so that
	those at $v_i$ and $v_j$ come last. Let $G_1$ be the graph obtained from
	$G$ by the replacements at all $v_k$ with $k\ne i,j$, and let
	$G_2=R_{v_i}(G_1)$, so that $\mathcal{R}(G)=R_{v_j}(G_2)$. In each of
	the first $m-2$ replacements, $v_i$ and $v_j$ are old vertices different
	from the replaced vertex, so the first part of
	Theorem~\ref{thm:tetrahedral-replacement}(iv) gives
	$\dist_{G_1}(v_i,v_j)=\dist_G(v_i,v_j)$. Applying the second part of
	Theorem~\ref{thm:tetrahedral-replacement}(iv) to $G_2=R_{v_i}(G_1)$
	with $q=v_j$ gives $\dist_{G_2}(w_i,v_j)=\dist_{G_1}(v_i,v_j)+1$, and
	applying it to $\mathcal{R}(G)=R_{v_j}(G_2)$ with $q=w_i$ gives
	$\dist_{\mathcal{R}(G)}(w_j,w_i)=\dist_{G_2}(v_j,w_i)+1$. Combining
	these three equalities yields (c).
\end{proof}

\begin{proof}[Proof of Theorem~\ref{thm:main}]
	Let $G_0$ be the $\frac{3}{2}$-tough plane triangulation with no
	2-factor constructed in~\cite{Shan2025}. Since a Hamiltonian cycle is a
	2-factor and $K_4$ is Hamiltonian, $G_0$ is not Hamiltonian and
	$G_0\ne K_4$. Moreover, since $G_0$ has no 2-factor,
	Theorem~\ref{thm:twofactor} implies that $G_0$ has two distinct vertices
	of degree $3$ at distance at most $2$; in particular,
	$D_3(G_0)\ne\emptyset$. Write $D_3(G_0)=\{v_1,\ldots,v_m\}$, where
	$m\ge 1$.

	Define $G^{(0)}=G_0$ and $G^{(s+1)}=\mathcal{R}(G^{(s)})$ for $s\ge 0$.
	For $1\le i\le m$, let $v_i^{(0)}=v_i$ and, for $s\ge 0$, let
	$v_i^{(s+1)}$ be the successor of $v_i^{(s)}$ in $G^{(s+1)}$. We claim
	that for every $s\ge 0$:
	\begin{enumerate}[(1)]
		\item $G^{(s)}$ is a non-Hamiltonian $\frac{3}{2}$-tough plane
		triangulation, and $G^{(s)}\ne K_4$;
		\item $D_3(G^{(s)})=\{v_1^{(s)},\ldots,v_m^{(s)}\}$; and
		\item $\dist_{G^{(s)}}(v_i^{(s)},v_j^{(s)})=\dist_{G_0}(v_i,v_j)+2s$
		for all distinct $i,j\in\{1,\ldots,m\}$.
	\end{enumerate}
	This holds for $s=0$. If it holds for some $s\ge 0$, then $\mathcal{R}$
	is applicable to $G^{(s)}$ and, by Lemma~\ref{lem:iterate}, (1)--(3)
	hold for $s+1$; here $G^{(s+1)}\ne K_4$ because $G^{(s+1)}$ is not
	Hamiltonian. This proves the claim.

	Now let $r=\lceil \ell/2\rceil$, so that $r\ge 1$ and $2+2r\ge \ell+2$,
	and let $G_\ell=G^{(r)}$. Let $u$ and $v$ be distinct vertices of degree
	$3$ in $G_\ell$. By (2), $u=v_i^{(r)}$ and $v=v_j^{(r)}$ for some
	distinct $i$ and $j$, and by (3) and Lemma~\ref{lem:independent},
	\[
	\dist_{G_\ell}(u,v)=\dist_{G_0}(v_i,v_j)+2r\ge 2+2r\ge \ell+2.
	\]
	This proves (i). Since $2+2r\ge 4$, Theorem~\ref{thm:twofactor} shows
	that $G_\ell$ has a 2-factor, which is (ii). Statement (iii) is part of
	(1).

	Finally, $G_\ell$ is $\frac{3}{2}$-tough by (1). Since $G_\ell\ne K_4$
	is a plane triangulation containing a vertex $u$ of degree $3$, we have
	$|V(G_\ell)|\ge 5$, so $G_\ell-N_{G_\ell}(u)$ has $u$ as an isolated
	vertex and contains at least one further vertex. Hence
	$c(G_\ell-N_{G_\ell}(u))\ge 2$ and
	\[
	\tau(G_\ell)\le \frac{|N_{G_\ell}(u)|}{c(G_\ell-N_{G_\ell}(u))}
	\le \frac{3}{2}.
	\]
	Therefore $\tau(G_\ell)=\frac{3}{2}$, which completes the proof.
\end{proof}

\section{Proof of Theorem~\ref{thm:tetrahedral-replacement}}
\label{sec:replacement}

Throughout this section, let $G$, $v$, $a,b,c$, and $x,y,z,w$ be as in the
definition of the tetrahedral replacement, and put
\[
H=R_v(G),\qquad A=\{x,y,z,w\},\qquad B=\{a,b,c\}.
\]
Recall that $H-A=G-v$ and that $B$ induces a triangle in $G-v$. Every edge
of $H$ leaving $A$ ends in $B$; each vertex of $B$ has exactly two
neighbors in $A$, both of which lie in $\{x,y,z\}$; and $w$ has no neighbor
outside $A$.

\begin{proof}[Proof of Theorem~\ref{thm:tetrahedral-replacement}]
	\emph{(i)} The tetrahedral replacement replaces the three triangular
	faces $vabv$, $vbcv$, $vcav$ of $G$ by the nine triangular faces
	\[
	abxa,\quad bcyb,\quad cazc,\quad axza,\quad bxyb,\quad cyzc,\quad
	xywx,\quad yzwy,\quad zxwz
	\]
	of $H$, and it leaves every other face unchanged. Hence $H$ is a plane
	triangulation. For later use we record the degrees: $d_H(w)=3$,
	$d_H(x)=d_H(y)=d_H(z)=5$, $d_H(u)=d_G(u)+1$ for $u\in B$, and
	$d_H(q)=d_G(q)$ for every other vertex $q$ of $G$.

	\smallskip
	\emph{(ii)} Suppose that $G$ is $\frac{3}{2}$-tough, and let
	$S\subseteq V(H)$ with $d:=c(H-S)\ge 2$. Set
	\[
	S_0=S\setminus A,\qquad k=|S\cap A|,\qquad
	g=c(G-S_0),\qquad g'=c((G-v)-S_0),
	\]
	where $S_0$ is regarded as a subset of $V(G)\setminus\{v\}$. We show
	that $|S|\ge\frac{3}{2}d$.

	Since $B$ induces a triangle, the vertices of $B\setminus S_0$ all lie
	in one component of $(G-v)-S_0$. The graph $G-S_0$ is obtained from
	$(G-v)-S_0$ by adding $v$ and its edges to $B\setminus S_0$, so $g=g'$
	if $B\setminus S_0\ne\emptyset$, and $g=g'+1$ otherwise.

	Suppose first that $k\le 3$. Then $A\setminus S$ is nonempty. As
	$A$ induces a $K_4$, $H[A\setminus S]$ is connected. The graph $H-S$ is obtained from
	$(G-v)-S_0=(H-A)-S_0$ by adding the connected graph $H[A\setminus S]$
	together with its edges to $B\setminus S_0$. Hence $d\in\{g',g'+1\}$,
	and in particular
	\begin{equation}\label{eq:components}
		d\le g'+1\le g+1.
	\end{equation}
	If $k\le 1$, then at most one vertex of $\{x,y,z\}$ lies in $S$, so
	every vertex of $B\setminus S_0$ still has a neighbor in $A\setminus S$.
	Thus $H[A\setminus S]$ is attached to the component of $(G-v)-S_0$
	containing $B\setminus S_0$ if $B\setminus S_0\ne\emptyset$, and forms a
	component of $H-S$ by itself otherwise; in either case $d=g$. Then
	$g=d\ge 2$, and the $\frac{3}{2}$-toughness of $G$ gives
	\[
	|S|\ge |S_0|\ge \tfrac{3}{2}g=\tfrac{3}{2}d.
	\]
	Now let $k\in\{2,3\}$. If $g\ge 2$, then by \eqref{eq:components} and
	the $\frac{3}{2}$-toughness of $G$,
	\[
	|S|=|S_0|+k\ge \tfrac{3}{2}g+k\ge \tfrac{3}{2}(g+1)\ge \tfrac{3}{2}d.
	\]
	If $g=1$, then \eqref{eq:components} and $d\ge 2$ give $d=2$, so we
	need $|S|\ge 3$. This is immediate if $k=3$. Let $k=2$. If
	$S_0=\emptyset$, then $H-S$ is connected: indeed, $G-v$ is connected
	because $G$ is connected and $B$ induces a triangle, and $H[A\setminus S]$
	is connected and contains a vertex of $\{x,y,z\}$, which has a neighbor
	in $B$. This contradicts $d=2$. Hence $S_0\ne\emptyset$, and
	$|S|=|S_0|+2\ge 3$.

	Finally, let $k=4$. Then $H-S=(G-v)-S_0=G-(S_0\cup\{v\})$, and so
	$c(G-(S_0\cup\{v\}))=d\ge 2$. By the $\frac{3}{2}$-toughness of $G$,
	$|S_0|+1\ge\frac{3}{2}d$, and therefore $|S|=|S_0|+4\ge\frac{3}{2}d$.
	This proves (ii).

	\smallskip
	\emph{(iii)} First suppose that $G$ has a Hamiltonian cycle $C$. Then
	$C$ uses exactly two of the three edges incident with $v$; by the
	symmetry of the construction, we may assume that these are $av$ and
	$vb$. Replacing the subpath $avb$ of $C$ by the path $azwxyb$ yields a
	Hamiltonian cycle of $H$.

	Conversely, let $C$ be a Hamiltonian cycle of $H$, and let $r$ be the
	number of edges of $C$ with exactly one end in $A$. Summing the degrees in $C$ of the
	four vertices of $A$ gives
	\begin{equation}\label{eq:cutedges}
		8=2|E(C[A])|+r.
	\end{equation}
	Since $w$ has no neighbor outside $A$, both edges of $C$ at $w$ belong
	to $C[A]$; hence $|E(C[A])|\ge 2$ and $r\le 4$. Moreover, $r$ is even
	by \eqref{eq:cutedges}, and $r>0$ because $C$ is connected and
	$V(H)\setminus A\ne\emptyset$. Therefore $r\in\{2,4\}$. Recall that
	every edge of $C$ leaving $A$ ends in $B$.

	Suppose that $r=2$. By \eqref{eq:cutedges}, $C[A]$ has three edges. As
	a proper subgraph of the cycle $C$, the graph $C[A]$ is acyclic; being
	a forest with four vertices, three edges, and maximum degree at most
	two, it is a path with vertex set $A$. The two ends of this path are joined
	by $C$ to vertices $p,q\in B$, which are distinct, since otherwise the
	path together with $p$ would form a cycle properly contained in $C$.
	Replacing the $(p,q)$-segment of $C$ through $A$ by the path $pvq$ gives
	a Hamiltonian cycle of $G$.

	Suppose that $r=4$. By \eqref{eq:cutedges}, $C[A]$ has exactly two
	edges, namely the two edges of $C$ at $w$. Thus $C[A]$ consists of a
	path $t_1wt_2$ and an isolated vertex $u$, where
	$\{u,t_1,t_2\}=\{x,y,z\}$. Both edges of $C$ at $u$ leave $A$, and
	since $H$ is simple they end at the two distinct neighbors of $u$ in
	$B$. The four edges of $C$ leaving $A$ end in the three vertices of
	$B$, so some $h\in B$ is incident with two of them, and these two are
	the edges of $C$ at $h$. They cannot both join $h$ to $u$, as $H$ is
	simple, and they cannot join $h$ to $t_1$ and to $t_2$, since then
	$ht_1wt_2h$ would be a cycle properly contained in $C$. Hence one of
	them joins $h$ to $u$, and the other joins $h$ to one of $t_1,t_2$, say
	to $t_1$. Consequently, $C[A\cup\{h\}]$ is the path $uht_1wt_2$. The
	remaining neighbor of $u$ on $C$ and the remaining neighbor of $t_2$ on
	$C$ lie in $B\setminus\{h\}$; call them $p$ and $q$, respectively. They
	are distinct, since otherwise $puht_1wt_2p$ would be a cycle properly
	contained in $C$. Thus $\{p,q\}=B\setminus\{h\}$, and $C$ contains the
	segment $puht_1wt_2q$. Since $ph$, $hv$, and $vq$ are edges of $G$,
	replacing this segment by the path $phvq$ gives a Hamiltonian cycle of
	$G$. The two cases $r=2$ and $r=4$ are illustrated in
	Figure~\ref{fig:traces}. This proves (iii).

	\begin{figure}[ht]
		\centering
		\begin{minipage}[c]{0.48\textwidth}
			\centering
			\begin{tikzpicture}[scale=0.90]
				\gadgetcoordinates
				\gadgetedges{gray!65}
				\draw[cycleedge] (a)--(z)--(w)--(x)--(y)--(b);
				\gadgetlabels
			\end{tikzpicture}

			\small (a) $r=2$.
		\end{minipage}
		\hfill
		\begin{minipage}[c]{0.48\textwidth}
			\centering
			\begin{tikzpicture}[scale=0.90]
				\gadgetcoordinates
				\gadgetedges{gray!65}
				\draw[cycleedge] (c)--(z)--(a)--(x)--(w)--(y)--(b);
				\gadgetlabels
			\end{tikzpicture}

			\small (b) $r=4$, with $h=a$, $u=z$, $p=c$, and $q=b$.
		\end{minipage}
		\caption{Representative segments of a Hamiltonian cycle of $H$
			through $A$ in the two cases of the proof of
			Theorem~\ref{thm:tetrahedral-replacement}(iii). Only the bold
			edges belong to the displayed segment.}
		\label{fig:traces}
	\end{figure}

	\smallskip
	\emph{(iv)} Let $p,q\in V(G)\setminus\{v\}$. If a $(p,q)$-path in $G$
	passes through $v$, then it contains a subpath $svs'$ with distinct
	$s,s'\in B$, and replacing this subpath by the edge $ss'$ yields a
	shorter $(p,q)$-path in $G-v$. Hence $\dist_G(p,q)=\dist_{G-v}(p,q)$.
	Similarly, if a $(p,q)$-path in $H$ meets $A$, then each of its maximal
	subpaths with all internal vertices in $A$ has two distinct ends in
	$B$, and replacing every such subpath by the edge joining its ends
	yields a shorter $(p,q)$-path in $H-A$. Hence
	$\dist_H(p,q)=\dist_{H-A}(p,q)$. Since $H-A=G-v$, we conclude that
	$\dist_H(p,q)=\dist_G(p,q)$.

	For the second statement, put
	$\dist_{G-v}(B,q)=\min_{u\in B}\dist_{G-v}(u,q)$. Every $(v,q)$-path in
	$G$ consists of an edge $vu$ with $u\in B$ followed by a $(u,q)$-path in
	$G-v$, and conversely every such concatenation is a $(v,q)$-path; hence
	$\dist_G(v,q)=1+\dist_{G-v}(B,q)$. In $H$, the vertex $w$ has no
	neighbor in $B$, while each vertex of $B$ has a neighbor in
	$\{x,y,z\}=N_H(w)$; hence $\dist_H(w,u)=2$ for every $u\in B$. Let $Q$
	be a shortest $(w,q)$-path in $H$, and let $u$ be the first vertex of $Q$
	in $B$, which exists because $q\notin A$ and every edge leaving $A$
	ends in $B$. The subpath of $Q$ from $w$ to $u$ has length at least
	$\dist_H(w,u)=2$, and the subpath of $Q$ from $u$ to $q$ has length at
	least $\dist_H(u,q)=\dist_{G-v}(u,q)$ by the first statement. Hence
	$\dist_H(w,q)\ge 2+\dist_{G-v}(B,q)$. Conversely, choosing $u\in B$
	with $\dist_{G-v}(u,q)=\dist_{G-v}(B,q)$ and concatenating a $(w,u)$-path
	of length $2$ with a shortest $(u,q)$-path in $G-v$ gives a $(w,q)$-walk in
	$H$ of length $2+\dist_{G-v}(B,q)$. Therefore
	\[
	\dist_H(w,q)=2+\dist_{G-v}(B,q)=\dist_G(v,q)+1.
	\]
	This completes the proof of Theorem~\ref{thm:tetrahedral-replacement}.
\end{proof}

\section{Hamiltonian cycles versus 2-factors}\label{sec:twofactors}

We conclude by pointing out the essential difference between Hamiltonian
cycles and 2-factors with respect to the tetrahedral replacement: it is
exactly the connectedness of a Hamiltonian cycle. Let $G$, $v$,
$B=\{a,b,c\}$, $H=R_v(G)$, and $A=\{x,y,z,w\}$ be as in
Section~\ref{sec:replacement}. For a subgraph $F$ of $H$ and disjoint sets
$X,Y\subseteq V(H)$, let $E_F(X,Y)$ denote the set of edges of $F$ with one
end in $X$ and the other end in $Y$.

In the forward direction there is no difference between the two notions: a
2-factor of $G$ extends to a 2-factor of $H$ by replacing its path of
length two through $v$ with a path with vertex set $A$, exactly as in the proof
of Theorem~\ref{thm:tetrahedral-replacement}(iii). The difference appears
in the backward direction.

\begin{lemma}\label{lem:twofactortrace}
	Let $F$ be a 2-factor of $H$, and let $r=|E_F(A,V(H)\setminus A)|$.
	Then $r\in\{0,2,4\}$. Moreover, unless
	\begin{enumerate}[(1)]
		\item $r=0$, or
		\item $r=2$ and the two edges of $E_F(A,V(H)\setminus A)$ have the
		same end in $B$,
	\end{enumerate}
	there is a 2-factor $F'$ of $G$ that contains every edge of $F$ having
	no end in $A$.
\end{lemma}

\begin{proof}
	Summing the degrees in $F$ of the four vertices of $A$ gives
	$8=2|E(F[A])|+r$, so $r$ is even. Since both edges of $F$ at $w$ lie in
	$F[A]$, we have $|E(F[A])|\ge 2$ and hence $r\le 4$. Thus
	$r\in\{0,2,4\}$. Recall that every edge of $F$ leaving $A$ ends in $B$.

	Suppose that $r=2$ and that the two edges of $E_F(A,V(H)\setminus A)$
	end at distinct vertices $p,q\in B$. Delete $A$ from $F$, and add the
	vertex $v$ together with the edges $pv$ and $vq$. Every vertex of $G$
	has degree $2$ in the resulting spanning subgraph of $G$, which is
	therefore a 2-factor $F'$ of $G$ with the required property.

	Suppose that $r=4$. For $u\in B$, let $e_u=|E_F(\{u\},A)|$. Then
	$e_a+e_b+e_c=4$ and $e_u\le 2$ for each $u\in B$, so the nonzero
	entries of $(e_a,e_b,e_c)$ are, up to order, either $2,1,1$ or $2,2$.
	In the first case, say $e_p=2$ and $e_q=e_s=1$; delete $A$ from $F$,
	and add $v$ together with the edges $vp$, $vq$, and $ps$. In the second
	case, say $e_p=e_q=2$; delete $A$ from $F$, and add $v$ together with
	the edges $vp$, $vq$, and $pq$. In both cases the added edges are edges
	of $G$ that are not in $F$, because both edges of $F$ at $p$ go to $A$.
	Every vertex of $G-v$ keeps its degree $2$, and $v$ receives degree
	$2$. Hence the resulting graph is a 2-factor $F'$ of $G$ with the
	required property.
\end{proof}

The two exceptional traces in Lemma~\ref{lem:twofactortrace} are shown in
Figure~\ref{fig:twofactor-obstructions}. In case (1), $F[A]$ has four
edges, so it is a 4-cycle, and this 4-cycle is a component of $F$. In case
(2), let $b$ be the common end in $B$ of the two edges of $F$ leaving $A$;
their ends in $A$ are distinct since $H$ is simple. Then $F[A]$ has three
edges and degree sequence $(2,2,1,1)$, so $F[A]$ is a path with vertex set $A$,
and $F[A\cup\{b\}]$ is a 5-cycle that is a component of $F$.

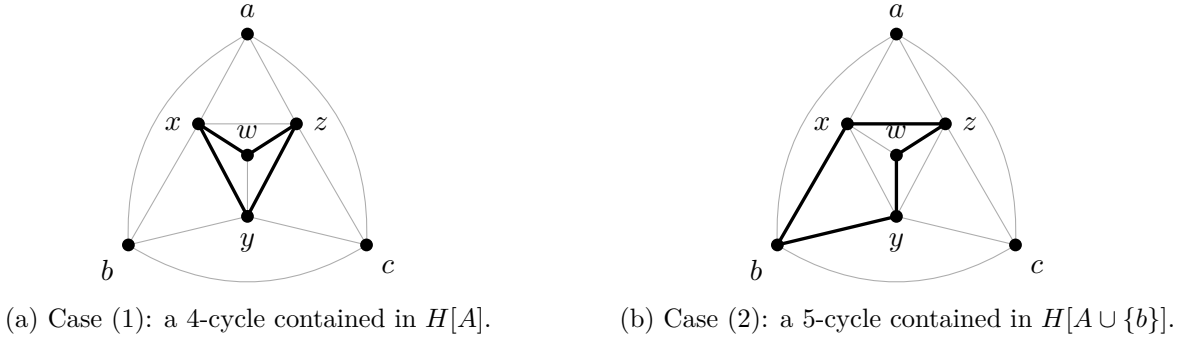
\begin{figure}[H]
	\centering
	\begin{minipage}[c]{0.48\textwidth}
		\centering
		\begin{tikzpicture}[scale=0.90]
			\gadgetcoordinates
			\gadgetedges{gray!65}
			\draw[cycleedge] (w)--(x)--(y)--(z)--cycle;
			\gadgetlabels
		\end{tikzpicture}

		\small (a) Case (1): a 4-cycle contained in $H[A]$.
	\end{minipage}
	\hfill
	\begin{minipage}[c]{0.48\textwidth}
		\centering
		\begin{tikzpicture}[scale=0.90]
			\gadgetcoordinates
			\gadgetedges{gray!65}
			\draw[cycleedge] (b)--(x)--(z)--(w)--(y)--cycle;
			\gadgetlabels
		\end{tikzpicture}

		\small (b) Case (2): a 5-cycle contained in $H[A\cup\{b\}]$.
	\end{minipage}
	\caption{Components of a 2-factor of $H$ that prevent the 2-factor from
		being pulled back through the tetrahedral replacement. Only the bold
		edges belong to the displayed component.}
	\label{fig:twofactor-obstructions}
\end{figure}

Neither trace can occur in a Hamiltonian cycle of $H$, since a Hamiltonian
cycle contains no cycle as a proper subgraph. Both can occur in a
2-factor, whose components need not be connected to one another. This is consistent with
Theorems~\ref{thm:twofactor} and~\ref{thm:main}: for the graph $G_0$ of
Section~\ref{sec:counterexample}, the graph $\mathcal{R}(G_0)$ has a
2-factor although $G_0$ has none, and since the reductions of
Lemma~\ref{lem:twofactortrace} at distinct replaced vertices do not
interfere with one another, every 2-factor of $\mathcal{R}(G_0)$ must
exhibit one of the two exceptional traces at some replaced vertex.
\section*{Declaration on the use of AI tools}

ChatGPT 5.6 was used to explore an initial construction for Theorem~\ref{thm:main}, which the author subsequently refined. Claude (Anthropic) was used to assist with language editing, grammar, clarity, and formatting, and to check the proofs. The author reviewed and verified all AI-assisted contributions and takes full responsibility for the content of the manuscript.

\end{document}